\documentclass[11pt,reqno]{amsart}

\usepackage{amsmath,amssymb,amsthm}
\usepackage[margin=1.15in]{geometry}
\usepackage[colorlinks=true,linkcolor=blue,citecolor=blue,urlcolor=blue]{hyperref}

\theoremstyle{plain}
\newtheorem{theorem}{Theorem}[section]
\newtheorem{proposition}[theorem]{Proposition}
\newtheorem{lemma}[theorem]{Lemma}
\newtheorem{corollary}[theorem]{Corollary}

\theoremstyle{definition}
\newtheorem{definition}[theorem]{Definition}
\newtheorem{example}[theorem]{Example}
\newtheorem{remark}[theorem]{Remark}

\DeclareMathOperator{\shdef}{shdef}
\DeclareMathOperator{\dsh}{dist_{sh}}
\newcommand{\dv}{d^{\mathrm{v}}}
\newcommand{\dvT}[1]{d^{\mathrm{v},\mathsf{T}}_{#1}}
\newcommand{\Lup}{L^{\mathrm{up}}_{k-1}}
\newcommand{\F}{\mathcal{F}}
\newcommand{\HH}{\mathcal{H}}
\newcommand{\Uu}{\mathcal{U}}
\newcommand{\M}{\mathcal{M}}
\newcommand{\Sij}{S_{ij}}

\begin{document}

\title[Stability of shifted complexes]{Stability of Shifted Complexes via the
Second-Moment Defect of the Up-Laplacian}

\author{Vinayak Gupta}
\address{Department of Mathematics, Indian Institute of Technology Guwahati,
Guwahati 781039, India}
\email{vinayakgupta1729v@gmail.com}


\begin{abstract}

Let $K$ be a finite pure $k$-dimensional simplicial complex, with $k\ge1$, on the vertex set $[n]$
and with facet family $K_k$. Let $\lambda_1(K)\ge\lambda_2(K)\ge\cdots>0$ be the nonzero
eigenvalues of its
$(k-1)$-dimensional up-Laplacian, and, after ordering the vertices so that
$\deg_K(1)\ge\cdots\ge\deg_K(n)$, let $\dvT{r}(K)$ be the number of vertices contained in at least
$r$ facets. A complex is \emph{shifted} if replacing a vertex of a face by a smaller vertex outside
the face always yields another face. We prove that there is a shifted family $\HH$ of
$(k+1)$-element subsets of $[n]$, with the same number of members as $K_k$, such that
\[
\tfrac12\bigl|K_k\,\triangle\,\HH\bigr|
\;\le\;
\tfrac12\left[\sum_{r\ge1}\bigl(\dvT{r}(K)\bigr)^{2}-\sum_{r}\lambda_r(K)^{2}\right].
\]
The left-hand side counts the facets that have to be exchanged to reach $\HH$; thus one half of the
gap between the second power sums of the two sequences bounds the distance of $K_k$ to a shifted
family. The characterization
$\lambda(K)=\dv(K)^{\mathsf T}\iff K$ is isomorphic to a shifted complex was established in
\cite{Gupta} from the identity that this gap equals twice the number of failed elementary shifts.
The present paper converts that identity into a quantitative stability statement and recovers the
equality characterization at zero defect. For $k=1$ this bounds the number of edge exchanges needed
to reach a threshold graph with the same number of edges.
\end{abstract}

\maketitle 

\section{Introduction}

A simplicial complex on the ordered vertex set $[n]=\{1,\dots,n\}$ is \emph{shifted} if replacing a
vertex of a face by a smaller vertex always produces another face whenever the replacement is
possible. For a pure complex, it is enough to check this condition on the facets. In dimension one,
shifted complexes are precisely threshold graphs, up to a relabeling of the vertices.

The spectral significance of shiftedness was established by Duval and Reiner
\cite{DuvalReiner}. They proved that the nonzero spectrum of the relevant simplicial Laplacian of a
shifted complex is equal to the conjugate vertex-degree partition. In the graph case, the
corresponding spectral identity for threshold graphs was proved by Merris \cite{Merris}. Duval and
Reiner also conjectured a majorization statement for arbitrary uniform families, with equality
precisely for shifted families. Its majorization part has since been disproved in every dimension at
least two \cite{Huang,ZSF}.

Let $K$ be a finite pure $k$-dimensional simplicial complex, with $k\ge1$. Write $\lambda(K)$ for
the nonzero
eigenvalues of its $(k-1)$-dimensional up-Laplacian and $\dv(K)^{\mathsf T}$ for the conjugate
vertex-degree partition. The equality question left by the Duval--Reiner conjecture was settled in
\cite{Gupta}:
\begin{equation}\label{eq:equality}
\lambda(K)=\dv(K)^{\mathsf T}
\quad\Longleftrightarrow\quad
K\text{ is isomorphic to a shifted complex}.
\end{equation}
In fact, Section 3 of \cite{Gupta} proved a stronger second-moment statement. Order the vertices by
nonincreasing degree. For $i<j$, let $N_{ij}$ count the facets for which replacing $j$ by $i$ is
possible, and let $R_{ij}$ count those for which the replacement produces another facet. Then
\[
\shdef(K)=\sum_{i<j}(N_{ij}-R_{ij})
\]
is the total number of failed elementary shifts, and
\begin{equation}\label{eq:defect}
\sum_{r\ge1}\bigl(\dvT{r}(K)\bigr)^2-
\sum_r\lambda_r(K)^2=2\shdef(K).
\end{equation}
Thus equality of the second power sums alone characterizes shiftedness. This is an exact
qualitative statement: it determines when the defect is zero, but does not say how many facets must
be changed when the defect is positive.

The purpose of this paper is to supply that quantitative statement. We work first with an
$r$-uniform family $\F\subseteq\binom{[n]}r$ on a fixed ambient set. Its distance to shiftedness is
\[
\dsh(\F)=\min\Bigl\{\tfrac12|\F\,\triangle\,\HH|\ :\
\HH\subseteq\tbinom{[n]}{r}\ \text{shifted},\ |\HH|=|\F|\Bigr\},
\]
the number of members that must be exchanged to obtain a shifted family of the same size.

\begin{theorem}\label{thm:introA}
Let $r\ge2$ and $\F\subseteq\binom{[n]}{r}$. Then there is a shifted family
$\HH\subseteq\binom{[n]}{r}$ with
$|\HH|=|\F|$ and
\[
\tfrac12\bigl|\F\,\triangle\,\HH\bigr|\ \le\ \shdef(\F).
\]
In particular $\dsh(\F)\le\shdef(\F)$.
\end{theorem}

The theorem converts the local count of failed shifts into a global repair bound. Although the
failures may interact, at most $\shdef(\F)$ members need to be replaced to obtain a shifted family.

Applying Theorem \ref{thm:introA} to the facet family $K_k$ and using \eqref{eq:defect} gives the
following result.

\begin{theorem}\label{thm:introB}
Let $K$ be a finite pure $k$-dimensional simplicial complex, with $k\ge1$, whose vertex set $[n]$
is labeled so that
$\deg_K(1)\ge\cdots\ge\deg_K(n)$. Then there is a shifted $(k+1)$-uniform family
$\HH\subseteq\binom{[n]}{k+1}$ with $|\HH|=|K_k|$ and
\[
\tfrac12\bigl|K_k\,\triangle\,\HH\bigr|
\ \le\ \tfrac12\left[\sum_{r\ge1}\bigl(\dvT{r}(K)\bigr)^{2}-\sum_{r}\lambda_r(K)^{2}\right].
\]
\end{theorem}

The left-hand side is the number of facets that must be exchanged to reach the shifted family
$\HH$. Hence one half of the second-moment gap bounds the distance to shiftedness. When the gap is
zero, the result reduces to the earlier characterization \eqref{eq:equality}. For $k=1$, it gives a
quantitative threshold-graph statement: the same gap bounds the number of edges that must be
exchanged to obtain a threshold graph with the same number of edges; see Corollary
\ref{cor:graph}.

The target family $\HH$ may omit labels of $[n]$ that occur in $K$; see Example
\ref{ex:vanishing} and Remark \ref{rmk:ambient}. Thus, for positive defect, $\HH$ need not be the
facet family of a pure complex on the same active vertex set as $K$.

\section{Uniform families, elementary shifts and the defect}\label{sec:prelim}

Throughout, $[n]=\{1,2,\dots,n\}$ is a fixed ambient set of vertex labels and $r\ge 2$ is an integer.
A family $\F\subseteq\binom{[n]}{r}$ is called $r$-\emph{uniform}; its members are $r$-element
subsets of $[n]$. We write $m=|\F|$ and, for $v\in[n]$,
\[
d_v(\F)=\bigl|\{F\in\F: v\in F\}\bigr|
\]
for the degree of $v$ in $\F$. In the application of Section \ref{sec:spectral} we take
$r=k+1$ and $\F=K_k$, the family of facets of a pure $k$-dimensional simplicial complex.

Two set-theoretic conventions are in force and should be kept apart. The ambient set $[n]$ is fixed
once and for all, whereas the \emph{active} set $\bigcup_{F\in\F}F$ of labels actually occurring in
$\F$ is not; the operations below preserve the former but may shrink the latter.

\begin{definition}[Elementary shift]\label{def:shift}
For $i<j$ and $F\in\binom{[n]}{r}$ with $j\in F$ and $i\notin F$, put
\[
\Sij(F)=(F\setminus\{j\})\cup\{i\}.
\]
We say the shift of $F$ from $j$ to $i$ is \emph{admissible}; it \emph{succeeds} if
$\Sij(F)\in\F$ and \emph{fails} otherwise. The family $\F$ is \emph{shifted} if no admissible shift
fails, that is, if $\Sij(F)\in\F$ whenever $F\in\F$, $j\in F$, $i<j$ and $i\notin F$.
\end{definition}

It is convenient to view $\Sij$ as the restriction of a permutation. Let $\tau=\tau_{ij}$ be the
transposition of $[n]$ exchanging $i$ and $j$, acting on subsets in the obvious way. If $j\in F$ and
$i\notin F$ then $\Sij(F)=\tau(F)$, and $\tau$ is an involution and a bijection of $\binom{[n]}{r}$
preserving all intersection cardinalities.

\begin{definition}[Admissible, successful and failed shifts]\label{def:NR}
For $i<j$ set
\[
N_{ij}(\F)=\bigl|\{F\in\F:\ j\in F,\ i\notin F\}\bigr|,\qquad
R_{ij}(\F)=\bigl|\{F\in\F:\ j\in F,\ i\notin F,\ \Sij(F)\in\F\}\bigr|,
\]
so that $N_{ij}(\F)-R_{ij}(\F)\ge0$ is the number of failed shifts from $j$ to $i$. The
\emph{shifting defect} of $\F$ is
\[
\shdef(\F)=\sum_{i<j}\bigl(N_{ij}(\F)-R_{ij}(\F)\bigr).
\]
We suppress the argument $\F$ from $N_{ij}$ and $R_{ij}$ when it is clear from the context.
\end{definition}

By construction $\shdef(\F)$ is a nonnegative integer, and
\begin{equation}\label{eq:defzero}
\shdef(\F)=0\iff \F\ \text{is shifted}.
\end{equation}
Both quantities depend on the linear order of $[n]$, as does shiftedness itself.

\begin{definition}[Adjacency count]\label{def:A}
Let $A(\F)$ be the number of unordered pairs $\{F,G\}$ of distinct members of $\F$ with
$|F\cap G|=r-1$, that is, of pairs differing in exactly one element.
\end{definition}

For $r=3$ the pair $\{123,124\}$ is counted by $A$ while $\{123,145\}$ is not. When $\F=K_k$ is the
facet family of a pure $k$-dimensional complex, $A(\F)$ is the number of unordered pairs of facets
sharing a ridge, the quantity denoted $A(K)$ in \cite[Proposition 3.8]{Gupta}.

\begin{definition}[Distance to shiftedness]\label{def:dist}
For $\F\subseteq\binom{[n]}{r}$ put
\[
\dsh(\F)=\min\Bigl\{\tfrac12|\F\,\triangle\,\HH|\ :\
\HH\subseteq\tbinom{[n]}{r}\ \text{shifted},\ |\HH|=|\F|\Bigr\}.
\]
\end{definition}

If $|\F|=|\HH|$ then $|\F\setminus\HH|=|\HH\setminus\F|=\tfrac12|\F\triangle\HH|$, so
$\dsh(\F)$ is exactly the least number of members that must be deleted from $\F$, and replaced by
the same number of new ones, in order to arrive at a shifted family. The minimum is taken over
shifted families inside the fixed ambient set $[n]$; a label of $[n]$ is permitted to occur in no
member of $\HH$. The following example shows why.

\begin{example}\label{ex:vanishing}
Let $\F=\{234,235\}\subseteq\binom{[5]}{3}$. Both members admit the shift $2\mapsto1$, and both
shifted sets $134,135$ lie outside $\F$. Replacing each member by its shift produces
$\{134,135\}$, a family of the same size in which the label $2$ occurs in no member. Thus an
operation that repairs failed shifts while preserving the number of members can nonetheless shrink
the active vertex set.
\end{example}

\section{The defect formula}\label{sec:formula}

The defect admits a closed formula in which the two summands of Definition \ref{def:NR} are replaced
by a weighted degree sum and the adjacency count. This formula is the bookkeeping device on which
everything below rests: it makes the effect of a compression on $\shdef$ transparent, because a
compression alters the degrees in a completely explicit way and can only increase $A$.

\begin{lemma}\label{lem:formula}
For every $\F\subseteq\binom{[n]}{r}$ with $|\F|=m$,
\[
\shdef(\F)=\sum_{v=1}^{n}(v-1)\,d_v(\F)-\binom{r}{2}m-A(\F).
\]
\end{lemma}

\begin{proof}
We compute $\sum_{i<j}N_{ij}$ and $\sum_{i<j}R_{ij}$ separately.

Fix $i<j$ and put $c_{ij}=|\{F\in\F: i,j\in F\}|$. Of the $d_j$ members containing $j$, exactly
$c_{ij}$ contain $i$ as well, so
\[
N_{ij}=d_j-c_{ij}.
\]
Summing over all pairs $i<j$, and noting that for fixed $j$ there are exactly $j-1$ admissible
values of $i$,
\[
\sum_{i<j}N_{ij}=\sum_{j=1}^{n}(j-1)d_j-\sum_{i<j}c_{ij}.
\]
Each member of $\F$ contains exactly $\binom{r}{2}$ unordered pairs of labels, and $\sum_{i<j}c_{ij}$
counts precisely these incidences, so $\sum_{i<j}c_{ij}=\binom{r}{2}m$ and
\begin{equation}\label{eq:sumN}
\sum_{i<j}N_{ij}=\sum_{j=1}^{n}(j-1)d_j-\binom{r}{2}m.
\end{equation}

Next consider $\sum_{i<j}R_{ij}$. Let $\{F,G\}$ be an unordered pair of distinct members of $\F$
with $|F\cap G|=r-1$. Then $F\triangle G=\{i,j\}$ for a unique pair $i<j$, and exactly one of the two
members contains $j$ and not $i$; that member is counted once by $R_{ij}$, and the other member is
not counted by $R_{ij}$ since it does not contain $j$. No other $R_{i'j'}$ sees the pair, because
$i,j$ are determined by $F\triangle G$. Conversely, a member counted by $R_{ij}$ together with its
successful shift forms a pair of distinct members differing in exactly one label. This correspondence
is a bijection, so
\begin{equation}\label{eq:sumR}
\sum_{i<j}R_{ij}=A(\F).
\end{equation}
Subtracting \eqref{eq:sumR} from \eqref{eq:sumN} gives the assertion.
\end{proof}

\section{Compression}\label{sec:compression}

\begin{definition}[Full compression]\label{def:compression}
Fix $i<j$ and let $\F\subseteq\binom{[n]}{r}$. Put
\[
\M=\M_{ij}(\F)=\{F\in\F:\ j\in F,\ i\notin F,\ \tau(F)\notin\F\},\qquad
\Uu=\F\setminus\M,
\]
where $\tau=\tau_{ij}$, and define the \emph{compression}
\[
C_{ij}\F=\Uu\cup\tau(\M).
\]
Thus a member is replaced by its shift exactly when that shift fails, and is left alone otherwise.
\end{definition}

Leaving a member alone when its shift already lies in $\F$ is what makes the operation
size-preserving: if the members with successful shifts were moved as well, two distinct members
would collapse onto the same set.

\begin{lemma}\label{lem:basic}
Let $i<j$ and write $t=|\M|$. Then
\begin{enumerate}
\item[(a)] $t=N_{ij}(\F)-R_{ij}(\F)$;
\item[(b)] $\tau(\M)\cap\F=\emptyset$; in particular $\tau(\M)\cap\Uu=\emptyset$ and
$|C_{ij}\F|=|\F|$;
\item[(c)] writing $d'_v$ for the degrees in $C_{ij}\F$, one has $d'_i=d_i+t$, $d'_j=d_j-t$ and
$d'_v=d_v$ for $v\ne i,j$;
\item[(d)] $\tfrac12\bigl|\F\,\triangle\,C_{ij}\F\bigr|=t$.
\end{enumerate}
\end{lemma}

\begin{proof}
(a) is Definition \ref{def:NR} read against Definition \ref{def:compression}: the members of $\M$
are exactly those on which the shift $j\mapsto i$ is admissible and fails. (b) holds because
$F\in\M$ means $\tau(F)\notin\F$, and $\Uu\subseteq\F$; since $\tau$ is injective, the union
defining $C_{ij}\F$ is disjoint and has $|\Uu|+|\M|=|\F|$ members. For (c), each member of $\M$
loses the label $j$ and gains the label $i$ and is otherwise unchanged, while members of $\Uu$ are
untouched. For (d), the members removed are those of $\M$ and the members inserted are those of
$\tau(\M)$; these two sets are disjoint by (b), and each member of $\M$ contains $j$ while each
member of $\tau(\M)$ does not, so no removed member is reinserted.
\end{proof}

We call the members of $\tau(\M)$ \emph{new} and those of $\Uu$ \emph{old}; by Lemma
\ref{lem:basic}(b) every member of $C_{ij}\F$ is of exactly one of these two kinds. The next lemma
is the heart of the argument.

\begin{lemma}\label{lem:Amonotone}
For every $i<j$ and every $\F\subseteq\binom{[n]}{r}$,
\[
A\bigl(C_{ij}\F\bigr)\ \ge\ A(\F).
\]
\end{lemma}

\begin{proof}
Retain the notation $\M,\Uu,\tau$ of Definition \ref{def:compression} and write
$\F'=C_{ij}\F=\Uu\cup\tau(\M)$. We construct an injection $\Phi$ from the set of adjacent pairs of
$\F$ into the set of adjacent pairs of $\F'$, where a pair is \emph{adjacent} when its two members
are distinct and meet in $r-1$ points. Let $\{F,G\}$ be an adjacent pair of $\F$. Since
$\F=\Uu\sqcup\M$, there are three cases.

\smallskip
\noindent\textit{Case 1: $F,G\in\Uu$.} Both members survive the compression, so
$\Phi(\{F,G\})=\{F,G\}$ is an adjacent pair of $\F'$ consisting of two old members.

\smallskip
\noindent\textit{Case 2: $F,G\in\M$.} Both are replaced, and
$|\tau(F)\cap\tau(G)|=|F\cap G|=r-1$ because $\tau$ is a bijection of $[n]$. Moreover
$\tau(F)\ne\tau(G)$. So $\Phi(\{F,G\})=\{\tau(F),\tau(G)\}$ is an adjacent pair of $\F'$ consisting
of two new members.

\smallskip
\noindent\textit{Case 3: $F\in\M$ and $G\in\Uu$.} Write $F=X\cup\{j\}$ with $i,j\notin X$. We
distinguish three sub-cases according to how $G$ meets $\{i,j\}$.

\smallskip
\noindent\textit{(3a): $G$ contains both of $i,j$ or neither of them.} Then $\tau(G)=G$, so
\[
|\tau(F)\cap G|=|\tau(F)\cap\tau(G)|=|F\cap G|=r-1 .
\]
Also $\tau(F)\ne G$, since $\tau(F)\notin\F$ while $G\in\F$. As $G\in\Uu$ survives, we may set
$\Phi(\{F,G\})=\{\tau(F),G\}$, an adjacent pair of $\F'$ with exactly one new member.

\smallskip
\noindent\textit{(3b): $j\in G$ and $i\notin G$.} Since $G\in\Uu$ was not moved, the only possible
reason is $\tau(G)\in\F$. The set $\tau(G)$ contains $i$ and not $j$, hence $\tau(G)\notin\M$ and
therefore $\tau(G)\in\Uu$; in particular $\tau(G)$ is an old member of $\F'$. As in Case 2,
$|\tau(F)\cap\tau(G)|=r-1$ and $\tau(F)\ne\tau(G)$, so we may set
$\Phi(\{F,G\})=\{\tau(F),\tau(G)\}$, again an adjacent pair of $\F'$ with exactly one new member.

\smallskip
\noindent\textit{(3c): $i\in G$ and $j\notin G$.} We claim this cannot occur. Write $G=Y\cup\{i\}$
with $i,j\notin Y$, so that $|X|=|Y|=r-1$. Since $j\notin G$ and $i\notin F$, we get
$F\cap G=X\cap Y$, and $|F\cap G|=r-1$ forces $X=Y$. Hence $G=X\cup\{i\}=\tau(F)$. But $F\in\M$
means $\tau(F)\notin\F$, contradicting $G\in\F$.

\smallskip
This defines $\Phi$ on all adjacent pairs of $\F$, and every image is an adjacent pair of $\F'$. It
remains to prove that $\Phi$ is injective. By Lemma \ref{lem:basic}(b) the old and new members of
$\F'$ are distinguishable, and the images produced in Cases 1, 2 and 3 contain respectively zero,
two and exactly one new member. Hence no two pairs from different cases have the same image, and it
suffices to check injectivity within each case.

In Case 1 the map is the identity. In Case 2 it is induced by the bijection $\tau$, hence injective.
In Case 3 an image has the form $\{\tau(F),Z\}$ with $\tau(F)$ new and $Z\in\Uu$ old; since $\tau$ is
injective, the new entry determines $F$. It determines the old entry too: in sub-case (3a) we have
$Z=G$, which contains both of $i,j$ or neither, whereas in sub-case (3b) we have $Z=\tau(G)$, which
contains $i$ but not $j$. These two possibilities are mutually exclusive, so $Z$ determines which
sub-case occurred, and then $G=Z$ or $G=\tau(Z)$ accordingly. Thus $\Phi$ is injective on Case 3 as
well, and $A(\F')\ge A(\F)$.
\end{proof}

\begin{proposition}\label{prop:drop}
Let $i<j$ and let $t=|\M_{ij}(\F)|$ be the number of members moved by $C_{ij}$. Then
\[
\shdef\bigl(C_{ij}\F\bigr)\ \le\ \shdef(\F)-(j-i)t\ \le\ \shdef(\F)-t .
\]
\end{proposition}

\begin{proof}
Write $\F'=C_{ij}\F$ and apply Lemma \ref{lem:formula} to $\F$ and to $\F'$. By Lemma
\ref{lem:basic}(c) the weighted degree sum changes by
\[
\sum_{v=1}^{n}(v-1)d'_v-\sum_{v=1}^{n}(v-1)d_v=(i-1)t-(j-1)t=-(j-i)t,
\]
by Lemma \ref{lem:basic}(b) the term $\binom{r}{2}|\F'|$ is unchanged, and by Lemma
\ref{lem:Amonotone} the term $-A(\F')$ is at most $-A(\F)$. Hence
$\shdef(\F')\le\shdef(\F)-(j-i)t$. The second inequality follows from $j-i\ge1$ and $t\ge0$.
\end{proof}

\begin{definition}[Weight]\label{def:weight}
The \emph{weight} of $\F\subseteq\binom{[n]}{r}$ is
$W(\F)=\sum_{F\in\F}\sum_{v\in F}v$.
\end{definition}

\begin{lemma}\label{lem:termination}
Call a compression \emph{nontrivial} when it moves at least one member. If
$t=|\M_{ij}(\F)|$, then
\[
W(C_{ij}\F)=W(\F)-(j-i)t.
\]
Consequently any sequence of nontrivial compressions is finite, and a family admitting no
nontrivial compression is shifted.
\end{lemma}

\begin{proof}
Each moved member exchanges the label $j$ for the label $i$ and is otherwise unchanged, so its
contribution to $W$ drops by $j-i$; unmoved members contribute the same as before. Since $W(\F)$ is
a nonnegative integer and each nontrivial compression strictly decreases it, no infinite sequence
of nontrivial compressions exists. Finally, if no compression is nontrivial
then $\M_{ij}(\F)=\emptyset$ for all $i<j$, that is, every admissible elementary shift succeeds, and
$\F$ is shifted by Definition \ref{def:shift}.
\end{proof}

\section{The stability theorem}\label{sec:stability}

\begin{theorem}\label{thm:stability}
Let $r\ge2$ and let $\F\subseteq\binom{[n]}{r}$ be an $r$-uniform family on the fixed ambient set
$[n]$. Then there exists a shifted family $\HH\subseteq\binom{[n]}{r}$ with $|\HH|=|\F|$ such that
\[
\tfrac12\bigl|\F\,\triangle\,\HH\bigr|\ \le\ W(\F)-W(\HH)\ \le\ \shdef(\F).
\]
In particular
\[
\dsh(\F)\ \le\ \shdef(\F).
\]
Some labels of the ambient set $[n]$ are allowed to occur in no member of $\HH$.
\end{theorem}

\begin{proof}
Put $\F_0=\F$ and apply nontrivial compressions repeatedly: given $\F_q$, if it is not shifted then
by \eqref{eq:defzero} some admissible shift fails, so some $C_{i_qj_q}$ is nontrivial, and we set
$\F_{q+1}=C_{i_qj_q}\F_q$. By Lemma \ref{lem:termination} the process stops after finitely many
steps, say at $\F_s=\HH$, and $\HH$ is shifted. By Lemma \ref{lem:basic}(b) every step preserves the
number of members, so $|\HH|=|\F|$.

Let $t_q\ge1$ be the number of members moved at step $q$ and put $w_q=(j_q-i_q)t_q$. Proposition
\ref{prop:drop} gives
\[
w_q\ \le\ \shdef(\F_q)-\shdef(\F_{q+1})\qquad(0\le q<s),
\]
and summing this telescoping bound, using $\shdef(\HH)=0$ from \eqref{eq:defzero}, yields
\begin{equation}\label{eq:telescope}
\sum_{q=0}^{s-1}w_q\ \le\ \shdef(\F_0)-\shdef(\F_s)=\shdef(\F).
\end{equation}
By Lemma \ref{lem:termination} the weight drops by exactly $w_q$ at step $q$, so
$W(\F)-W(\HH)=\sum_{q}w_q$, which with \eqref{eq:telescope} gives the second inequality of the
theorem.

For the first, Lemma \ref{lem:basic}(d) gives $\tfrac12|\F_q\triangle\F_{q+1}|=t_q$, and the
triangle inequality for the symmetric difference, applied along the chain
$\F_0,\F_1,\dots,\F_s$, gives
\[
\tfrac12\bigl|\F\,\triangle\,\HH\bigr|\ \le\ \sum_{q=0}^{s-1}\tfrac12\bigl|\F_q\triangle\F_{q+1}\bigr|
=\sum_{q=0}^{s-1}t_q\ \le\ \sum_{q=0}^{s-1}w_q = W(\F)-W(\HH),
\]
the last inequality because $j_q-i_q\ge1$. Finally $\dsh(\F)\le\tfrac12|\F\triangle\HH|$ by
Definition \ref{def:dist}.
\end{proof}

\begin{remark}\label{rmk:sharp}
The inequality $\dsh(\F)\le\shdef(\F)$ is sharp. Here and below we abbreviate a set by the concatenation of its labels. For $r=2$ and
$\F=\{13\}\subseteq\binom{[3]}{2}$ the
only admissible shift is $3\mapsto2$ and it fails, so $\shdef(\F)=1$; and $\F$ is not shifted, so
$\dsh(\F)=1$. Similarly $\F=\{124\}\subseteq\binom{[5]}{3}$ has $\shdef(\F)=\dsh(\F)=1$.
\end{remark}

\begin{remark}\label{rmk:gap}
The two quantities are nevertheless of different orders in general. If $\F=\{F\}$ is a single set
$F=\{v_1<\cdots<v_r\}$, then every admissible shift fails, so
\[
\shdef(\F)=\sum_{j\in F}\bigl|\{i<j:\ i\notin F\}\bigr|=\sum_{t=1}^{r}(v_t-t),
\]
which is unbounded, whereas $\dsh(\F)\le1$ because $\{\{1,\dots,r\}\}$ is shifted. Thus Theorem
\ref{thm:stability} is far from tight for families with a few badly placed members; it is designed
to be tight in the opposite regime, where the defect is small.
\end{remark}

\section{The spectral form}\label{sec:spectral}

Let $K$ be a finite pure $k$-dimensional simplicial complex, with $k\ge1$, on $[n]$. Write $K_k$
for its facet family, $f_k(K)=|K_k|$, $\dvT{r}(K)$ for the number of vertices contained in at least $r$ facets,
and $\lambda_r(K)$ for the nonzero eigenvalues of $\Lup(K)$. After labeling the vertices by
nonincreasing degree, put $\shdef(K)=\shdef(K_k)$.

\begin{theorem}[{\cite[Theorem 3.9]{Gupta}}]\label{thm:import}
Let $K$ be a finite pure $k$-dimensional simplicial complex, with $k\ge1$, whose vertices are
ordered by nonincreasing degree. Then
\[
\sum_{r\ge1}\bigl(\dvT{r}(K)\bigr)^{2}-\sum_{r}\lambda_r(K)^{2}=2\shdef(K).
\]
\end{theorem}

\begin{theorem}\label{thm:spectral}
Let $K$ be a finite pure $k$-dimensional simplicial complex, with $k\ge1$, on $[n]$. If its vertices
are labeled by nonincreasing degree, then there is a shifted family
$\HH\subseteq\binom{[n]}{k+1}$ with $|\HH|=f_k(K)$ such that
\[
\tfrac12\bigl|K_k\,\triangle\,\HH\bigr|\ \le\ \shdef(K)
=\ \tfrac12\left[\sum_{r\ge1}\bigl(\dvT{r}(K)\bigr)^{2}-\sum_{r}\lambda_r(K)^{2}\right].
\]
\end{theorem}

\begin{proof}
Apply Theorem \ref{thm:stability} to $K_k$ and then use Theorem \ref{thm:import}.
\end{proof}

\begin{corollary}[Threshold-graph stability]\label{cor:graph}
Let $G$ be a finite simple graph on the ambient vertex set $[n]$, labeled so that
$\deg_G(1)\ge\deg_G(2)\ge\cdots\ge\deg_G(n)$. Write
\[
d_r^{\mathsf T}(G)=\bigl|\{v\in[n]:\deg_G(v)\ge r\}\bigr|
\]
for its conjugate degree sequence, and let
$\mu_1(G)\ge\cdots\ge\mu_n(G)=0$ be its Laplacian eigenvalues. Then there is a threshold graph $H$
on $[n]$ with $|E(H)|=|E(G)|$ such that
\[
\tfrac12\bigl|E(G)\,\triangle\,E(H)\bigr|
\ \le\ \tfrac12\left[
\sum_{r\ge1}\bigl(d_r^{\mathsf T}(G)\bigr)^2
-\sum_{i=1}^{n}\mu_i(G)^2
\right].
\]
In particular, if the second-moment gap vanishes, then $G$ is a threshold graph.
\end{corollary}

\begin{proof}
Delete the isolated vertices, apply Theorem \ref{thm:spectral} with $k=1$, and then restore the
isolated vertices. Shifted $2$-uniform families are precisely threshold graphs, and isolated
vertices do not change either second power sum. The edgeless case is immediate.
\end{proof}

\begin{remark}\label{rmk:ambient}
The family $\HH$ lies in the fixed ambient set $[n]$, but it need not use every vertex occurring in
$K$; see Example \ref{ex:vanishing}. Thus $\HH$ need not be the facet family of a pure complex on
the same active vertex set as $K$.
\end{remark}

If the second-moment gap is zero, Theorem \ref{thm:spectral} gives $K_k=\HH$. Since $K$ is pure, a
shifted facet family implies that $K$ itself is shifted. This recovers
\cite[Theorem 3.12]{Gupta}.


\begin{thebibliography}{9}

\bibitem{DuvalReiner}
A.~M. Duval and V.~Reiner,
\emph{Shifted simplicial complexes are Laplacian integral},
Trans. Amer. Math. Soc. \textbf{354} (2002), 4313--4344.
\href{https://doi.org/10.1090/S0002-9947-02-03082-9}{DOI:10.1090/S0002-9947-02-03082-9}

\bibitem{Gupta}
V.~Gupta,
\emph{Spectral bounds and shifted complexes: eigenvalues of the up-Laplacian via face degrees},
arXiv:2608.01694 (2026). \href{https://arxiv.org/abs/2608.01694}{arXiv:2608.01694}

\bibitem{Huang}
J.~Huang,
\emph{The Duval--Reiner conjecture: counterexamples and the second partial-sum inequality},
arXiv:2607.20051 (2026). \href{https://arxiv.org/abs/2607.20051}{arXiv:2607.20051}

\bibitem{Merris}
R.~Merris,
\emph{Degree maximal graphs are Laplacian integral},
Linear Algebra Appl. \textbf{199} (1994), 381--389.
\href{https://doi.org/10.1016/0024-3795(94)90361-1}{DOI:10.1016/0024-3795(94)90361-1}

\bibitem{ZSF}
H.-Z. Zhang, Y.-M. Song and Y.-Z. Fan,
\emph{Degree majorization and Laplacian eigenvalue sums for simplicial complexes},
arXiv:2607.20910 (2026). \href{https://arxiv.org/abs/2607.20910}{arXiv:2607.20910}

\end{thebibliography}
\end{document}